\documentclass{article}

\usepackage{amsmath,amsthm} 
\usepackage{amssymb,mathrsfs} 
\usepackage{a4wide} 
\usepackage{graphicx}
\usepackage{physics}
\usepackage{color,subfigure} 
\usepackage{enumerate}
\usepackage{todonotes}
\usepackage[normalem]{ulem}
\usepackage{cancel}

\newtheorem{theorem}{Theorem}

\newtheorem{definition}{Definition}
\newtheorem{proposition}{Proposition}

\DeclareMathAlphabet{\mathpzc}{OT1}{pzc}{m}{it}
\newcommand{\dps}{\displaystyle } 
\newcommand{\rme}{\mathrm{e}}
 
\newcommand{\cL}{\mathcal{L}}
\newcommand{\cLs}{\mathcal{S}}
\newcommand{\cLa}{\mathcal{A}}

\newcommand{\cLFD}{\mathcal{L}_{\rm FD}}
\newcommand{\Lovd}{\mathcal{L}_{\rm ovd}}
\newcommand{\cB}{\mathcal{B}}

\newcommand{\cD}{\mathcal{D}}

\newcommand{\R}{\mathbb{R}}

\newcommand{\cR}{\mathcal{R}}
\newcommand{\calH}{\mathcal{H}}
\newcommand{\cH}{\mathcal{H}}

\newcommand{\subplus}{\textnormal{\texttt{+}}}
\renewcommand{\leq}{\leqslant}
\renewcommand{\geq}{\geqslant}

\newcommand{\cE}{\mathcal{E}}
\newcommand{\LFD}{\mathcal{L}_{\rm FD}}
\newcommand{\Lham}{\mathcal{L}_{\rm ham}}
\newcommand{\Pinot}{\mathscr{P}}
\newcommand{\lang}{\left\langle}
\newcommand{\rang}{\right\rangle}
\renewcommand{\dim}{d}
\newcommand{\calL}{\mathcal{L}}
\newcommand{\scrD}{\mathscr{D}}
\newcommand{\Schur}{\mathfrak{S}_0}

\begin{document}

  
\title{Computational statistical physics and hypocoercivity}
\author{G. Stoltz \\
  {\small CERMICS, Ecole des Ponts, Marne-la-Vallée, France and MATHERIALS project-team, Inria Paris, France}
}
 
\maketitle

\begin{abstract}
  This note provides an introduction to molecular dynamics, the computational implementation of the theory of statistical physics. The discussion is focused on the properties of Langevin dynamics, a degenerate stochastic differential equation which can be seen as a perturbation of Hamiltonian dynamics. From an analytical point of view, the generator of Langevin dynamics is a degenerate elliptic operator. The evolution of the law of the stochastic process is governed by the Fokker-Planck equation, and its longtime convergence can be obtained via hypocoercive techniques, some of which are reviewed here. One consequence of these analytical results in terms of error estimates for the computation of average properties of molecular systems is the estimation of the asymptotic variance of time averages in a central limit theorem.
\end{abstract}

\section{A short introduction to computational statistical physics}

\subsection{Aims of molecular dynamics}

Molecular simulation has been used and developed over the past 70~years, and its number of users keeps increasing; see~\cite{AT17,FrenkelSmit,Tuckerman} for reference textbooks in the physics literature. As we understand it, it has two major aims nowadays. First, it can be used as a {\it numerical microscope}, which allows to perform ``computer'' experiments. This was the initial motivation for simulations at the microscopic level: physical theories were tested on computers. Today, understanding the behavior of matter at the microscopic level can still be difficult from an experimental viewpoint (because of the high resolution required, both in time and in space), or because we simply do not know what to look for! Numerical simulations are then a valuable tool to test some
ideas or obtain some data to process and analyze in order to help assessing experimental setups. Another major aim of molecular simulation, maybe even more important than the previous one, is to compute macroscopic quantities or thermodynamic properties, typically through averages of some functionals of the system. In this case, molecular simulation is a way to obtain \emph{quantitative} information on a system,
instead of resorting to approximate theories, constructed for simplified models, and giving only qualitative answers. More generally, molecular simulation is a tool to explore the links between the microscopic and macroscopic properties of a material, allowing to address modelling questions such as ``Which microscopic ingredients are necessary (and which are not) to observe a given macroscopic behavior?''

Physical systems are described at the microscopic level by their positions~$q \in \mathcal{D} = (L\mathbb{T})^d$ (with~$\mathbb{T} = \mathbb{R}\backslash \mathbb{Z}$ the one-dimensional torus) or $\mathbb{R}^d$, and momenta $p \in \mathbb{R}^{d}$. The associated phase-space is denoted by $\cE = \mathcal{D} \times \mathbb{R}^d$. The description of systems in statistical physics requires a fundamental
ingredient: microscopic interaction laws between the constituents of matter
and possibly the environment. The interactions between the particles are taken into
account through a potential function~$V$, depending on the positions~$q$ only.
The total energy of the system is given by the Hamiltonian
\begin{equation}
  \label{intro:Hamiltonian}
  H(q,p) = V(q) + \frac{1}{2} \, p^\top M^{-1} p,
\end{equation}
where~$M \in \mathbb{R}^{d \times d}$ is the mass matrix, and~$V:\mathbb{R}^d \to \mathbb{R}$ the potential energy function. The macroscopic state of a system is described, within the framework of
statistical physics, by a probability measure~$\mu$ on the
phase space $\cE = {\cal D} \times \mathbb{R}^{d}$. 
Macroscopic features of the system are then computed as
averages of an observable~$\varphi$ with respect to this measure:
\begin{equation}
  \label{intro:averages}
\mathbb{E}_\mu(\varphi) = \int_{\cE} \varphi(q,p) \, \mu(dq \, dp).  
\end{equation}
We therefore call the measure~$\mu$ the \emph{macroscopic state} of the system -- also know as thermodynamic ensemble.

In many physical situations, systems in contact with some energy thermostat
are considered, rather than isolated systems with a fixed energy.
In this case, the energy of the system fluctuates. The temperature (a notion which has in fact a precise definition in statistical physics) is however fixed. Microscopic configurations are then distributed according to the so-called \emph{canonical measure} 
\begin{equation}
  \label{intro:canonical}
  \mu(dq \, dp) = Z_\mu^{-1} \exp (-\beta H(q,p)) \, dq \, dp,
\end{equation}
where $\beta = 1/(k_{\rm B}T)$ ($T$ denotes the temperature and $k_{\rm{B}}$ the
Boltzmann constant), and~$Z_\mu$ is a normalization constant. The canonical measure is of the
tensorized form 
\[
\mu(dq \, dp) = \nu(dq) \, \kappa(dp),
\]
where 
\begin{equation}
  \label{intro:density}
  \nu(dq) = Z_\nu^{-1} \textrm{e}^{-\beta V(q) } \, d q,
  \qquad
  Z_\nu = \int_{\cal D} {\rm e}^{-\beta V(q)}\, dq,
\end{equation}
and~$\kappa$ is a Gaussian measure with covariance~$M/\beta$. Therefore, sampling configurations $(q,p)$ according to the canonical measure~$\mu(dq\,dp)$ can be performed by independently sampling positions according to~$\nu(dq)$ 
and momenta according to~$\kappa(dp)$. Since it is straightforward to sample from~$\kappa$, the actual issue is to sample from~$\nu$. 

\subsection{Computing average properties using stochastic differential equations}

The main mathematical challenge in computing ensemble averages such as~\eqref{intro:averages} is the very high dimensionality of the integral under consideration, which prevents the use of standard quadrature methods. In practice, the only realistic option is to rely on ergodic averages, where configurations are generated according to the probability measure~$\mu$ by integrating a dynamics, and the ensemble average of some observable~$\varphi \in L^1(\mu)$ is approximated as
\begin{equation}
  \label{eq:ergodic_avg}
  \int_\mathcal{E} \varphi \, d\mu = \lim_{t \to +\infty} \widehat{\varphi}_t, \qquad \widehat{\varphi}_t = \frac1t \int_0^t \varphi(q_s,p_s) \, ds.
\end{equation}

\subsubsection{Langevin dynamics}
We focus in these notes on the so-called Langevin dynamics
\begin{equation}
  \label{eq:Langevin}
  \left \{ \begin{aligned}
    d q_t & = M^{-1} p_t \, dt, \\
    d p_t & = - \nabla V (q_t) \, dt - \gamma \, M^{-1} p_t \, dt + \sqrt{\frac{2\gamma}{\beta}} \, dW_t,
    \end{aligned} \right.
\end{equation}
where~$W_t$ is a standard $d$-dimensional Brownian motion, and~$\gamma > 0$ the magnitude of the friction term. Note that the Langevin dynamics can be seen as a perturbation of the Hamiltonian dynamics, and in fact reduces to the Hamiltonian dynamics for~$\gamma=0$. There is a balance between the added friction term~$- \gamma \, M^{-1} p_t \, dt$ and the fluctuation term~$\sqrt{2\gamma \beta^{-1}} \, dW_t$: the magnitude of the fluctuation term is chosen so that the stochastic dynamics~\eqref{eq:Langevin} leaves the canonical measure~\eqref{intro:canonical} invariant.

In order to mathematically study the properties of the Langevin dynamics, it is useful to introduce some objects, in particular some differential operators. The time evolution of average properties is encoded by the semigroup
\[
\left(\rme^{t \cL}\varphi\right)(q,p) = \mathbb{E}\left[\varphi(q_t,p_t) \, \Big| (q_0,p_0)=(q,p)\right],
\]
with generator 
\begin{equation}
  \label{eq:generator_Langevin}
  \cL = \Lham + \gamma \LFD,
  \qquad
  \Lham = p^\top M^{-1}\nabla_q - \nabla V^\top \nabla_p,
  \qquad
  \LFD = - p^\top M^{-1}\nabla_p + \frac1\beta \Delta_p.
\end{equation}
The expression of~$\mathcal{L}$ is obtained from Itô calculus, which implies that
\begin{equation}
  \label{eq:derivative_averages}
\frac{d}{dt} \left(  \mathbb{E}\left[\varphi(q_t,p_t) \, \Big| (q_0,p_0)=(q,p)\right] \right) =  \mathbb{E}\left[(\cL \varphi)(q_t,p_t) \, \Big| (q_0,p_0)=(q,p)\right].
\end{equation}
The existence and uniqueness of the invariant measure is characterized by the property
\[
\forall \varphi \in C^\infty_{\rm c}(\cE), \qquad \int_\cE \cL \varphi \, d\mu = 0.
\]
The law of the process at time~$t$, denoted by~$\psi(t,q,p)$, evolves according to the Fokker--Planck equation
\begin{equation}
  \label{eq:FokkerPlanck}
  \partial_t \psi = \cL^\dagger \psi,
\end{equation}
where $\cL^\dagger$ is the adjoint of $\cL$ on $L^2(\cE)$. This equation is formally obtained by noting that~\eqref{eq:derivative_averages} can be rewritten as 
\[
\frac{d}{dt}\left(\int_\cE \varphi \, \psi(t) \right) = \int_\cE (\cL \varphi) \, \psi(t).
\]

From a functional analytical viewpoint, it is in fact more convenient to work in $L^2(\mu)$. In order to do so, we introduce~$f(t) = \psi(t)/\mu$, and denote the adjoint of~$\cL$ on $L^2(\mu)$ by $\cL^*$. In this functional framework,
\[
\cL^* = -\Lham + \gamma \LFD,
\qquad
\LFD = -\frac1\beta \sum_{i=1}^d \partial_{p_i}^* \partial_{p_i},
\qquad
\Lham = \frac{1}{\beta} \sum_{i=1}^d \partial_{p_i}^* \partial_{q_i} - \partial_{q_i}^* \partial_{p_i}.
\]
Indeed, a simple computation gives  
\[
\int_{\cD} \left(\partial_{q_i} \varphi\right) \phi \, d\mu = -\int_{\cD} \varphi \left(\partial_{q_i} \phi\right) d\mu - \int_{\cD} \varphi \phi \, \partial_{q_i} \left(Z_\nu^{-1} \rme^{-\beta V} \right) d\kappa,
\]
so that $\partial_{q_i}^* = -\partial_{q_i} + \beta \partial_{q_i} V$. A similar computation gives $\partial_{p_i}^* = -\partial_{p_i} + \beta (M^{-1} p)_i$. In particular, the Fokker--Planck equation rewrites
\begin{equation}
  \label{eq:FP_reformulated}
  \partial_t f = \cL^* f.
\end{equation}
Given the structure of the operators at hand (only the sign of the Hamiltonian part changes when taking adjoints), convergence results for $\rme^{t \cL}$ on $L^2(\mu)$ are very similar to the ones for $\rme^{t \cL^*}$. Convergence results for $\rme^{t \cL^*}$ can in turn be transferred to convergence results for~\eqref{eq:FokkerPlanck}.

\subsubsection{Expected rates of convergence for Langevin dynamics}

It is important to understand how the rate of convergence of the Langevin dynamics depends on the friction parameter~$\gamma>0$, in order to tune this parameter in numerical simulations to have the fastest convergence. Let us first consider the Hamiltonian (or underdamped) limit $\gamma\to 0$. A simple computation using It\^o calculus gives
\[
\frac{d}{dt}\mathbb{E}\left[H(q_t,p_t)\right] = -\gamma\left( \mathbb{E}\left[p_t^\top M^{-2} p_t\right] -\frac1\beta\mathrm{Tr}(M^{-1})\right)dt.
\]
This suggests that the typical time to change energy levels in this limit scales as $1/\gamma$; the same timescale allows in fact to sample the canonical measure, which is a function of the energy. Precise statements on the limit as~$\gamma \to 0$ of the diffusion process~$H(q_{t/\gamma},p_{t/\gamma})$, corresponding to the evolution of the energy accelerated by a factor~$1/\gamma$, are given in~\cite{HP08} for one dimensional systems.

In the overdamped limit $\gamma \to +\infty$, in the simple case when~$M = \mathrm{Id}$, it is possible to rewrite the dynamics in a suggestive form using a rescaling of time~$\gamma t$:
\[
\begin{aligned}
  q_{\gamma t} - q_0  & =  - \frac1\gamma \int_0^{\gamma t} \nabla V(q_s) \, d s + \sqrt{\frac{2}{\gamma\beta}} W_{\gamma t} - \frac1\gamma \left( p_{\gamma t} - p_0\right) \\
  & =  - \int_0^{t} \nabla V(q_{\gamma s}) \, d s + \sqrt{2\beta^{-1}} B_{t} - \frac1\gamma \left( p_{\gamma t} - p_0\right).
\end{aligned}
\]
This suggests that solutions to Langevin dynamics accelerated in time by a factor~$\gamma$ converge to solutions of the overdamped Langevin dynamics
\begin{equation}
  \label{eq:ovd}
  dQ_t = -\nabla V(Q_t)\, dt + \sqrt{\frac2\beta} \, dB_t.
  \end{equation}
From a functional analytical viewpoint, this limit is encoded in the following approximative equality, obtained by asymptotic analysis: $\mathrm{e}^{\gamma t (\Lham + \gamma \LFD)} \approx \mathrm{e}^{t \mathcal{L}_\mathrm{ovd}}$ with~$\mathcal{L}_\mathrm{ovd} = -\nabla V^\top \nabla_q + \beta^{-1} \Delta_q$.

The conclusion of the discussion of these two limiting regimes is that the convergence rate of Langevin dynamics should scale as~$\min(\gamma,\gamma^{-1})$. 

\subsubsection{Ergodicity results for Langevin dynamics}
The almost-sure convergence of the ergodic averages~$\widehat{\varphi}_t$ in~\eqref{eq:ergodic_avg} follow from he results of~\cite{Kli87} since the stochastic dynamics preserves a probability measure, and its generator is hypoelliptic~\cite{Hor67}. The asymptotic variance of these ergodic averages allows to quantify the statistical error:
\begin{equation}
  \label{eq:asymptotic_variance}
  \lim_{t \to +\infty} \mathrm{Var}\left[\widehat{\varphi}_t^2\right]
  = 2 \int_\cE \int_0^{+\infty}\!\! \left(\rme^{t \mathcal{L}}\Pinot \varphi\right) \Pinot \varphi \, dt \, d\mu
  = 2 \int_\cE \left(-\mathcal{L}^{-1}\Pinot \varphi\right) \Pinot \varphi \, d\mu
\end{equation}
where
\begin{equation}
  \label{eq:def_Pi_mu}
  \Pinot \varphi = \varphi - \mathbb{E}_\mu(\varphi),
\end{equation}
and where we used the following operator equality
\begin{equation}
  \label{eq:inv_by_int}
  -\mathcal{L}^{-1} = \int_0^{+\infty} \rme^{t \mathcal{L}} \, dt
\end{equation}
on the Hilbert space
\begin{equation}
  \label{eq:L_2_0}
  L^2_0(\mu) = \Pinot L^2(\mu) = \left\{ \varphi \in L^2(\mu) \, \left| \int_\cE \varphi \, d\mu = 0 \right. \right\}.
\end{equation}
This is legitimae when the operator norm of the semigroup~$\rme^{t\mathcal{L}}$ decays sufficiently fast, for intance exponentially. In fact, in such a setting, the Poisson equation 
\[
-\mathcal{L} \Phi = \Pinot \varphi = \varphi - \int_\mathcal{E} \varphi \, d\mu
\]
has a solution in $L^2(\mu)$, and so a central limit theorem holds~\cite{Bhattacharya}.

There are various techniques for obtaining the exponential convergence of the semigroup in Banach subspaces of~$L^2_0(\mu)$, besides the techniques we describe more precisely in Section~\ref{sec:Langevin} (see also the introduction of~\cite{BFLS20} for an extensive review):
\begin{itemize}
\item a first approach is based on Lyapunov techniques~\cite{Wu01,MSH02,rey-bellet,HM11}, which rely on convergence estimates in the Banach space
  \[
  B^\infty_\mathcal{K}(\mathcal{E}) = \left\{ \varphi \, \textrm{measurable}, \, \sup \left|\frac{\varphi}{\mathcal{K}}\right| <+\infty \right\},
  \]
  where~$\mathcal{K} : \mathcal{E} \to [1,+\infty]$ is a Lyapunov function, \emph{i.e.} a function such that~$\mathcal{L} \mathcal{K} \leq -a \mathcal{K} + b$ for some constants~$a>0$ and~$b \in \mathbb{R}$;
\item the hypocoercive framework~$H^1(\mu)$ was popularized by the monograph~\cite{Villani09}, which was building on various previous works where (iterated) commutators of the hypoelliptic generator were key tools to obtain the longtime convergence~\cite{Talay02,EH03,HN04}; 
\item convergence results in~$H^1(\mu)$ can be transferred to~$L^2(\mu)$ after hypoelliptic regularization~\cite{Herau07};
\item it was recently shown how to directly obtain convergence in~$L_0^2(\mu)$ in~\cite{AM19,CLW19,Brigati21}, based on a space-time Poincar\'e inequality involving the operator~$\partial_t - \mathcal{L}_{\rm ham}$;
\item finally, exponential convergence can also be obtained from purely probabilistic arguments based on a clever coupling between two realizations of the stochastic differential equation~\eqref{eq:Langevin}, see~\cite{EGZ19}. 
\end{itemize}
  
\section{Longtime convergence of overdamped Langevin dynamics}
\label{sec:ovd}

Before considering the more complicated case of Langevin dynamics in Section~\ref{sec:Langevin}, we start by studying the longtime convergence of overdamped Langevin dynamics~\eqref{eq:ovd}, whose generator reads 
\begin{equation}
  \label{eq:Lovd}
  \Lovd = -\nabla V(q)\cdot \nabla_q + \frac1\beta \Delta_q = -\frac1\beta \sum_{i=1}^d \partial_{q_i}^* \partial_{q_i},
\end{equation}
where adjoints in the last term are taken on~$L^2(\nu)$. It is clear from the last equality that~$\Lovd$ is a symmetric operator; in fact it is even self-adjoint~\cite{bakry-gentil-ledoux-14}. The function~$\varphi(t)=\rme^{t \Lovd}\varphi_0$ is a solution to the partial differential equation $\partial_t \varphi(t) = \Lovd \varphi(t)$. This equation preserves mass since
\[
\frac{d}{dt}\left( \int_{\cD} \varphi(t) \, d\nu \right) =  \int_{\cD} \Lovd\varphi(t) \, d\nu = \int_{\cD} \varphi(t) \left(\Lovd\mathbf{1} \right)d\nu = 0.
\]
This suggests the following longtime limit:
\[
\varphi(t) \xrightarrow[t\to+\infty]{} \int_{\cD} \varphi_0 \, d\nu.
\]
Assuming without loss of generality that~$\varphi$ belongs to the subspace $L^2_0(\nu) \subset L^2(\nu)$ of functions with average~0 with respect to~$\nu$, we are led to proving that~$\varphi(t)$ converges to~0 in~$L^2(\mu)$. We compute to this end 
\begin{equation}
  \label{eq:decay_Lovd}
\frac{d}{dt}\left(\frac12 \left\|\varphi(t)\right\|^2_{L^2(\nu)} \right) = \langle \Lovd \varphi(t), \varphi(t) \rangle_{L^2(\nu)} = -\frac1\beta \left\|\nabla_q \varphi(t) \right\|_{L^2(\nu)}^2 \leq 0.
\end{equation}
This inequality already shows that~$\left\|\varphi(t)\right\|_{L^2(\nu)}$ is non-increasing.

To prove the convergence of~$\left\|\varphi(t)\right\|_{L^2(\nu)}$ to~0, and to obtain a rate of convergence, we need further assumptions. The simplest setting to consider is when a so-called Poincar\'e inequality holds; see for instance~\cite[Chapter~4]{bakry-gentil-ledoux-14} for a very nice introduction to these inequalities.

\begin{definition}[Poincar\'{e} inequality]
  Consider the functional spaces
  \[
  L_0^2(\nu) = \left\{ \varphi \in L^2(\nu) \, \left|\, \int_\cD \varphi \, d\nu = 0 \right.\right\},
  \qquad
  H^1(\nu) = \left\{ \varphi \in L^2(\nu) \, \left|\, \nabla \varphi \in ( L^2(\nu) )^d \right. \right\}.
  \]
  The measure~$\nu$ is said to satisfy a Poincar\'e inequality with constant~$R > 0$ when
  \begin{equation}
    \label{eq:Poincare_inequality}
    \forall \varphi \in H^1(\nu) \cap L^2_0(\nu),
    \qquad 
    \| \varphi \|^2_{L^2(\nu)} \leq \frac{1}{R} \| \nabla \varphi \|^2_{L^2(\nu)}. 
  \end{equation}
\end{definition}

The constant $R > 0$ depends on the potential~$V$, the inverse temperature~$\beta$ and the domain~$\cD$. Various sufficient conditions for~$\nu$ to satisfy a Poincar\'e inequality are for instance reviewed in~\cite[Section~2.2]{LS16}. One simple example of sufficient condition is that the potential energy~$V$ can be written as the sum~$V_{\rm uc} + V_{\rm b}$ of a uniformly convex part (\emph{i.e.} $\nabla^2 V_{\rm uc} \geq \alpha \mathrm{Id}_d$ with~$\alpha>0$) and a bounded perturbation~$V_{\rm b} \in L^\infty(\cE)$. The inequality~\eqref{eq:Poincare_inequality} implies (and is in fact equivalent to) the exponential convergence to~0 of the semigroup $\rme^{t \Lovd}$ considered as an operator on $L_0^2(\nu)$. 

\begin{proposition}\label{prop:IP}
 The measure $\nu$ satisfies a Poincar\'e inequality with constant~$R > 0$ if and only if 
\begin{equation}
  \label{eq:exp_decrease_semigroup}
  \left\| \rme^{t \Lovd}  \right\|_{\cB(L_0^2(\nu))} \leq \rme^{-R t/\beta}.
\end{equation}
\end{proposition}

\begin{proof}
Let us first assume that the measure $\nu$ satisfies a Poincar\'e inequality with constant~$R > 0$. For $\varphi \in D(\Lovd) \cap L_0^2(\nu)$,
\begin{equation}
\label{eq:spectral_gap}
 -  \langle \Lovd \varphi, \varphi  \rangle_{L^2(\nu)} = \frac1\beta \| \nabla \varphi \|^2_{L^2(\nu)} \geq \frac{R}{\beta} \| \varphi \|_{L^2(\nu)}^2.
\end{equation}
Since 0 is an eigenvalue of the operator~$\Lovd$ (whose associated eigenvectors are constant functions), this inequality shows that the spectral gap of the self-adjoint operator $-\Lovd$ on $L^2(\nu)$ is larger than or equal to~$R/\beta$ (using a Raylegh--Ritz principle).
The inequality \eqref{eq:spectral_gap} also gives the exponential decrease of the semigroup on $L_0^2(\nu)$ since, using~\eqref{eq:decay_Lovd} and~\eqref{eq:Poincare_inequality},
\begin{equation}
  \label{eq:derivative_L2_norm_squared}
  \frac{d}{d t}\left(\frac12  \left\| \rme^{t \Lovd} \varphi \right\|^2_{L^2(\nu)} \right) =  \left\langle \rme^{t \Lovd} \varphi, \Lovd \,\rme^{t \Lovd} \varphi \right\rangle_{L^2(\nu)} \leq - \frac{R}{\beta} \left\| \rme^{t \Lovd} \varphi \right\|^2_{L^2(\nu)}.
\end{equation}
By a Gronwall inequality, it follows that~$\left\| \rme^{t \Lovd}\varphi  \right\|_{L_0^2(\nu)} \leq \rme^{-R t/\beta} \| \varphi \|_{L_0^2(\nu)}$. The bound is finally extended to all functions in~$L^2_0(\nu)$ by density.

Assume now that $ \| \rme^{t \Lovd}  \|_{\cB(L_0^2(\nu))} \leq \rme^{-R t/\beta}$. Then, for a given $\varphi \in L^2_0(\nu)$ and any $t > 0$,
\[
\frac{ \left\| \rme^{t \Lovd}\varphi \right\|_{L^2_0(\nu)}^2 - \left\| \varphi \right\|_{L^2_0(\nu)}^2 }{t} \leq \| \varphi \|_{L^2_0(\nu)}^2 \, \frac{\rme^{-2R t/\beta} - 1}{t}.
\]
We next restrict ourselves to a $C^\infty$ function with compact support, and pass to the limit $t \to 0$, using the equalities in \eqref{eq:spectral_gap} and \eqref{eq:derivative_L2_norm_squared}:
\[
-\frac2\beta \| \nabla \varphi \|^2_{L^2(\nu)} \leq -\frac{2R}{\beta}\| \varphi \|_{L^2_0(\nu)}^2.
\]
The Poincar\'e inequality \eqref{eq:Poincare_inequality} finally follows by a density argument.
\end{proof}

Let us emphasize that, crucially, the prefactor for the exponential convergence in~\eqref{eq:exp_decrease_semigroup} is~1. Note also that the convergence rate is not degraded when one adds to the generator an antisymmetric part~$\cLa = F\cdot \nabla$ with $\mathrm{div}(F \rme^{-\beta V}) = 0$ since~\eqref{eq:spectral_gap} still holds with~$\Lovd$ replaced by~$\Lovd+\cLa$.

\section{Longtime convergence of hypocoercive ordinary differential equations}

We present in this section the spirit of hypocoercive methods, illustrated on the possibly simplest example, namely the two dimensional ordinary differential equation $\dot{X} = L X \in \mathbb{R}^2$ with, for $\gamma > 0$,
\[
-L = A + \gamma S, \qquad A = \begin{pmatrix} 0 & 1 \\ -1 & 0 \end{pmatrix}, \qquad S = \begin{pmatrix} 0 & 0 \\ 0 & 1 \end{pmatrix}.
\]
Solutions of this ordinary differential equation are presented in Figure~\ref{fig:ODE}. The structure of $-L$ has features typical of hypocoercive operators: the symmetric part~$S$ is positive but degenerate, while the antisymmetric part~$A$ couples the kernel and the image of~$S$. This simple example can be treated analytically: a straightforward computation reveals that the smallest real part of the eigenvalues of~$-L$ (the spectral gap) is of order~$\min(\gamma,\gamma^{-1})$. In fact, the determinant of~$-L$ is~1, the trace~$\gamma$, so the eigenvalues are
\[
\lambda_\pm = \frac\gamma2 \pm \left(\frac{\gamma^2}{4}-1\right)^{1/2},
\]
where the square root is understood as a complex number when~$\gamma \leq 2$.
To obtain the longtime convergence of $\rme^{tL}$, we can diagonalize the evolution operator as
\[
\rme^{t L} = U^{-1} \begin{pmatrix} \rme^{-t \lambda_+} & 0 \\ 0 & \rme^{-t\lambda_-} \end{pmatrix} U.
\]
The decay rate is provided by the spectral gap~$\lambda = \min\{\mathrm{Re}(\lambda_-), \mathrm{Re}(\lambda_+)\}$, equal to~$\gamma/2$ for~$\gamma \leq 2$, and to $2/(\gamma + \sqrt{\gamma^2-4})$ for~$\gamma \geq 2$:
\begin{equation}
  \label{eq:ODE_decay}
  |X(t)| = \left|\rme^{t L} X(0)\right| \leq C \rme^{-\lambda t}|X(0)|.
\end{equation}

\begin{figure}
  \begin{center}
  \includegraphics[width=0.5\textwidth]{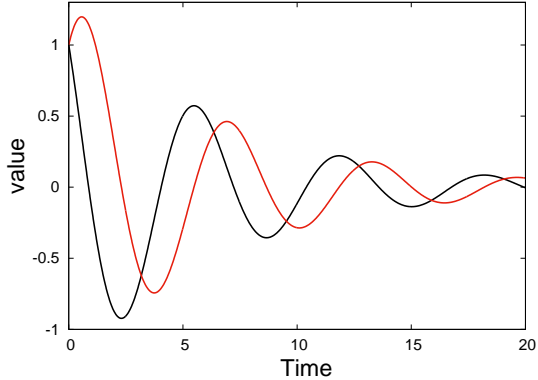}
  \end{center}
\caption{\label{fig:ODE} Values of $X_1(t)$ (black) and~$X_2(t)$ (red) for $X(0)=(1,1)$ and $\gamma=0.5$.}
\end{figure}

Let us now try to prove the decay estimate~\eqref{eq:ODE_decay} using the finite dimensional counterpart of the method used in Section~\ref{sec:ovd}. The starting point is the decay equality 
\[
\frac{d}{dt}\left(\frac12 |X(t)|^2 \right) = -\gamma X(t)^\top S X(t) = -\gamma X_2(t)^2 \leq 0.
\]
Note that, although the derivative of the squared norm is nonpositive, it is not negative, because dissipation in the~$X_1$ variable is missing. We cannot therefore conclude from this equality that $|X(t)|$ converges to~0. The key idea is to change the scalar product with some positive definite matrix~$P$:
\[
|X|^2_P = X^\top P X, \qquad \frac{d}{dt}\left(|X(t)|_P^2 \right) = X(t)^\top \left(PL+L^\top P\right) X(t).
\]
We then choose~$P$ to couple~$X_1$ and~$X_2$. We start perturbatively, by adding a small term to the identity matrix as
\[
P = \mathrm{Id} - \varepsilon \begin{pmatrix}0 & 1 \\ 1 & 0 \end{pmatrix}.
\]
In this case, 
\[
-\left(PL+L^\top P\right) = 2 \gamma P S+ 2 \varepsilon \begin{pmatrix} 1 & 0 \\ 0 & -1 \end{pmatrix} \approx 2\begin{pmatrix}\varepsilon & 0 \\ 0 & \gamma \end{pmatrix}.
\]
This provides some (small) dissipation in~$X_1$. An optimal choice for~$P$ is proposed in~\cite{AAS15}. It is based on the informal statement ``$L^\top P \geq \lambda P$'', which motivates constructing~$P$ from a diagonalization of~$L^\top$:
\[
P = a_- X_- \overline{X}_-^\top + a_+ X_+ \overline{X}_+^\top, \qquad a_\pm > 0, \qquad L^\top X_\pm = \lambda_\pm X_\pm.
\]
With this choice, $-(PL + L^\top P) \geq 2\lambda P$. Therefore, $|X(t)|_P^2 \leq \rme^{-2\lambda t}|X_0|_P^2$, and so, by equivalence of scalar products,
\[
|X(t)| \leq \min\left(1, C \rme^{-\lambda t}\right)|X_0|,
\]
which is consistent with~\eqref{eq:ODE_decay}. Note that the prefactor $C \geq 1$ in this inequality is really needed, since, similarly to Proposition~\ref{prop:IP}, it can easily be shown that exponential convergence holds with $C=1$ if and only if $-L$ is coercive (\emph{i.e.} $-X^\top LX \geq \alpha |X|^2$ with $\alpha>0$).

\section{Longtime convergence of Langevin dynamics}
\label{sec:Langevin}

\subsection{Lack of coercivity}

Solutions to the Fokker--Planck equation~\eqref{eq:FP_reformulated} are expected to converge to the constant function~$\mathbf{1}$. Upon subtracting this constant function from the initial condition~$f_0$, the convergence of the law amounts to the convergence to~0 of $\rme^{t \cL^*}(f_0 - \mathbf{1})$. This motivates again working on the space~$L^2_0(\mu)$ introduced in~\eqref{eq:L_2_0}. The same functional space is considered when studying the convergence of~$\rme^{t \cL}\varphi$ towards its limiting value~$\mathbb{E}_\mu(\varphi)$, since~$\rme^{t \cL}\varphi - \mathbb{E}_\mu(\varphi) = \rme^{t \cL}\Pinot \varphi$ (with~$\Pinot$ the projector defined in~\eqref{eq:def_Pi_mu}), so that it is sufficient to prove that~$\rme^{t \cL}\varphi$ converges to~0 for functions~$\varphi \in L^2_0(\mu)$. 

The important remark which motivates the title of this section is that the generator~$\cL$ of the Langevin dynamics~\eqref{eq:generator_Langevin} fails to be coercive on $L_0^2(\mu)$ since second derivatives in~$q$ are missing. In fact, for $C^\infty$ and compactly supported test functions~$\varphi$, we obtain that
\begin{equation}
  \label{eq:fails_to_be_coercive}
  -\langle \cL \varphi,\varphi\rangle_{L^2(\mu)} = \frac\gamma\beta \| \nabla_p \varphi \|^2_{L^2(\mu)},
\end{equation}
which should be compared to~\eqref{eq:decay_Lovd} for overdamped Langevin dynamics. The key idea of hypocoercivity is to modify the canonical~$L^2(\mu)$ scalar product to introduce some mixed derivatives in~$q$ and~$p$ in order to retrieve some dissipation in~$q$ through some commutator identities. This idea was already present in the computations performed in~\cite[Section~3]{Talay02}, and was later generalized in~\cite{Villani09}. This motivates the name for the technique in view of the analogy with hypoellipticity, since (more or less explicit) commutator identities allow to recover some form of coercivity for operators with degenerate diffusion parts, in the same way commutators identities in H\"ormander's theory~\cite{Hor67} imply hypoellipticity and therefore allow to recover regularity results for operators with degenerate diffusion parts similar to the regularity results for elliptic operators.

\subsection{An almost direct~$L^2$ approach}

We present in this section a way to prove the exponential decay of the semigroup $\rme^{t\cL}$ in $L^2(\mu)$, by modifying the scalar product with some operator involving the generator of the Hamiltonian part of the dynamics. This approach was first proposed in~\cite{Herau06} and then extended in~\cite{DMS09,DMS15}. It is more direct than first proving a decay estimate in~$H^1(\mu)$ and then transfering this decay to~$L^2(\mu)$ by hypoelliptic regularization (see~\cite{Villani09,Herau07} as well as the review of these approaches in~\cite{LS16}), or by some spectral argument (using the bounded self-adjoint operators $Q_t = \rme^{t\cL^*}\rme^{t\cL}$, as done in~\cite{DPD18} by resorting to~\cite[Lemma~2.9]{HSV14}). It also turns out to be more robust to perturbations, since it can be used for nonequilibrium systems in a perturbative framework~\cite{BHM17,IOS19} or for spectral discretization of the Langevin dynamics~\cite{RS18}. It also allows to quantify more easily the convergence rate in terms of the parameters of the dynamics, in particular the friction rate~\cite{DKMS13,GS16}.

As mentioned above, for notational simplicity, we study the convergence to~0 of $\rme^{t\cL}\varphi$ for $\varphi \in L^2_0(\mu)$ rather than the convergence to~0 of $\rme^{t\cL^*}(f_0-1)$ for $f_0 \in L^2(\mu)$.

\begin{theorem}[Hypocoercivity in $L^2(\mu)$]
  \label{th:hypocoercivity}
  Assume that $\nu(dq) = Z_\nu^{-1} \rme^{-\beta V(q)} \, dq$ satisfies the Poincar\'e inequality~\eqref{eq:Poincare_inequality} with a constant~$R_\nu>0$, and that~$V \in C^\infty(\cD)$ is such that there exist $c_1 > 0$, $c_2 \in [0,1]$ and $c_3 > 0$ for which 
  \begin{equation}
    \label{eq:regularization condition_d}
    \Delta V \leq c_1 \dim + \frac {c_2 \beta} 2 | \nabla V |^2, \qquad \left|\nabla^2 V \right|^2 = \sum_{i,j=1}^d \left|\partial_{q_i}\partial_{q_j}V\right|^2 \leq c_3^2 \left( \dim + | \nabla V |^2 \right).
  \end{equation}
  Then there exist $C > 1$ and $\lambda_\gamma > 0$ (which are explicitly computable in terms of the parameters of the dynamics, $C$ being independent of $\gamma>0$) such that, for any $\varphi \in L^2_0(\mu)$,
  \begin{equation}
    \label{eq:cv_expo_L2}
    \forall t \geq 0, \qquad \left\| \rme^{t \mathcal{L}} \varphi \right\|_{L^2(\mu)} \leq C \rme^{-\lambda_\gamma t} \| \varphi \|_{L^2(\mu)}.
  \end{equation}
  Moreover, the convergence rate is of order $\min(\gamma,\gamma^{-1})$: there exists $\overline{\lambda} > 0$ such that
  \[
  \lambda_\gamma \geq \overline{\lambda} \min(\gamma,\gamma^{-1}).
  \]
\end{theorem}

Note that some prefactor~$C > 1$ appears in~\eqref{eq:cv_expo_L2}, compared to the estimates~\eqref{eq:exp_decrease_semigroup} for overdamped Langevin dynamics. The scaling with respect to the dimension~$\dim$ of the constants in the bounds~\eqref{eq:regularization condition_d} is motivated by the case of separable potentials for which $V(q) = v(q_1) + \dots + v(q_d)$ for some smooth one dimensional function~$v$, which corresponds to tensorized probability measures. The bounds~\eqref{eq:regularization condition_d} then follow from the inequalities
\[
  v'' \leq c_1 + \frac{c_2 \beta}{2} (v')^2, \qquad \left|v''\right|^2 \leq c_3^2 \left( 1+\left|v'\right|^2 \right).
\]
These bounds generally hold if $v$ has polynomial growth for example. The scaling of the constants should be similar for particles on a lattice (such as one dimensional atom chains) with finite interaction ranges, or systems for which correlations between degrees of freedom are bounded with respect to the dimension, in the sense that each column/line of the matrix $\nabla^2 V$ has a finite number of nonzero entries.

The convergence result of Theorem~\ref{th:hypocoercivity} can be (formally) extended to more general Hamiltonian functions, in particular separable Hamiltonians $H(q,p)=V(q)+U(p)$ under appropriate assumptions on~$U$, namely some moment conditions for derivatives of~$U$ and a Poincar\'e inequality for the probability measure with density proportional to~$\rme^{-\beta U}$; see~\cite{ST18} for precise statements. Let us emphasize that we do not need the generator to be hypoelliptic, though, and can allow for instance for kinetic energy functions which vanish on open sets. In fact, Theorem~\ref{th:hypocoercivity} can be extended to certain Piecewise Deterministic Markov Processes, see~\cite{DMS09,DMS15,ADNR18}. 

Let us sketch the proof of Theorem~\ref{th:hypocoercivity}. The first step is to consider an appropriate change of scalar product, which uses the antisymmetric part~$\Lham$ of the generator. More precisely, we consider the modified squared norm
\[
\calH[\varphi] = \frac 1 2 \|\varphi\|_{L^2(\mu)}^2 - \varepsilon \lang R \varphi, \varphi \rang,
\qquad 
R = \Big( 1 + (\Lham \Pi_0)^* (\Lham \Pi_0) \Big)^{-1} (\Lham \Pi_0)^*, 
\]
where~$\Pi_0$ is the projector whose action is to integrate out the~$p$ variable according to~$\kappa$:
\[
(\Pi_0 \varphi)(q) = \int_{v \in \R^d} \varphi(q,p) \, \kappa(dp).
\]
It is shown in~\cite{DMS15} for instance that $R = \Pi_0 R (1-\Pi_0)$ and $\Lham R$ are bounded operators on~$L^2(\mu)$ (with bounds smaller than~$1/2$ and~1, respectively), and that the modified square norm~$\calH$ is equivalent to the standard squared norm~$\|\cdot\|_{L^2(\mu)}^2$ for $\varepsilon \in (-1,1)$. The motivation for the expression of the regularization operator~$R$ is that
\[
(\Lham \Pi_0)^* (\Lham \Pi_0) = \beta^{-1} \nabla_q^* \nabla_q
\]
is an operator coercive in the~$q$ variable, so that
\[
R \Lham \Pi_0 = \frac{(\Lham \Pi_0)^* (\Lham \Pi_0)}{1 + (\Lham \Pi_0)^* (\Lham \Pi_0)}
\]
is coercive as well when acting on functions of the position variable only. The coercivity comes here from spectral calculus and the Poincar\'e inequality for~$\nu$, which translates into the inequality $\nabla_q^* \nabla_q \geq K_\nu^2 \Pi_0$ in the sense of symmetric operators.

The key result to prove the coercivity of~$-\mathcal{L}$ in the modified scalar product~$\lang \lang\cdot,\cdot\rang\rang$ induced by~$\calH$ is the following dissipation inequality:
\begin{equation}
  \label{eq:dissipation}
\mathscr{D}[\varphi] := \lang \lang -\calL \varphi, \varphi \rang \rang \geq \lambda \|\varphi\|^2.
\end{equation}
Since
\[
\frac{d}{dt}\left(\calH\left[\rme^{t \calL}\varphi\right]\right) = -\scrD\left[\rme^{t \calL}\varphi\right] \leq -\frac{2\lambda}{1+\varepsilon} \calH\left[\rme^{t \calL}\varphi\right],
\]
a Gronwall inequality allows to conclude to the exponential convergence of~$\calH\left[\rme^{t \calL}\varphi\right]$ to~0, which in turn leads to the estimate~\eqref{eq:cv_expo_L2}. The proof of the dissipation inequality~\eqref{eq:dissipation} motivates why the regularization operator was chosen. Indeed, upon controlling remainder terms which are not explicitly written in the inequalities below (thanks to elliptic regularity estimates for the operator~$1+\nabla_q^*\nabla_q$ considered on~$L^2(\nu)$), 
  \[
  \begin{aligned}
    \mathscr{D}[\varphi] & = \gamma\lang -\LFD \varphi, \varphi \rang + \varepsilon \lang R \Lham \Pi_0 \varphi, \varphi \rang + \mathrm{O}(\gamma\varepsilon) \\
    & = \frac\gamma\beta \| \nabla_p \varphi \|_{L^2(\mu)}^2 + \varepsilon \lang \frac{\nabla_q^* \nabla_q}{\beta + \nabla_q^* \nabla_q}\Pi_0 \varphi, \Pi_0\varphi \rang + \mathrm{O}(\gamma\varepsilon) \\
    & \geq \frac{\gamma K^2_\kappa}{\beta} \| (1 - \Pi_0) \varphi \|_{L^2(\mu)}^2 + \frac{\varepsilon K_\nu^2}{\beta + K_\nu^2} \|\Pi_0 \varphi\|_{L^2(\mu)}^2 + \mathrm{O}(\gamma\varepsilon),
  \end{aligned}
  \]
where we used in the last step the operator inequality~$\nabla_p^* \nabla_p \geq K_\kappa^2 (1-\Pi_0)$, which is a direct consequence of the fact that~$\kappa$ satisfies a Poincaré inequality. The scaling of the exponential convergence rate finally follows from the fact that~$\varepsilon$ is of the order~$\min(\gamma,\gamma^{-1})$, a choice which follows from a careful inspection of the remainder terms.

\subsection{Directly obtaining bounds on the resolvent}

We present in this section the results of~\cite{BFLS20}, which allow to direcly obtain bounds on the generator~$\cL^{-1}$, and hence on the solutions to the Poisson equation~\eqref{eq:def_Pi_mu} and on the asymptotic variance~\eqref{eq:asymptotic_variance}. The starting point is to realize that typical hypocoercive operators on~$L^2_0(\mu)$ have a ``saddle-point like'' structure: 
\[
\cL = \begin{pmatrix} 0 & \cLa_{0\subplus}\\
  \cLa_{\subplus0} & \cL_{\subplus\subplus} \end{pmatrix}, \qquad \cH = \cH_0 \oplus \cH_{\subplus}, \qquad \cH_0 = \Pi_0 \cH, 
\]
where~$\cLa$ is the antisymmetric part of the generator (for Langevin dynamics, $\cLa = \cL_{\mathrm{ham}}$), and, denoting by $\Pi_{\subplus} = 1 - \Pi_{0}$ the orthogonal projector complementary to~$\Pi_{0}$, the operators~$T_{\alpha\beta} =  \Pi_{\alpha} T \Pi_{\beta}:\, \cH_{\beta} \to \cH_{\alpha}$ are the restrictions (blocks) of a given operator~$T$. The formal inverse of~$\cL$ can be written in terms of the Schur complement $\Schur = \cLa_{\subplus 0}^* \cL_{\subplus\subplus}^{-1} \cLa_{\subplus0}$ as
\[
\cL^{-1} =
\begin{pmatrix}
  \Schur^{-1} & -\Schur^{-1} \cLa_{0\subplus} \cL_{\subplus\subplus}^{-1}\\
  -\cL_{\subplus\subplus}^{-1} \cLa_{\subplus0} \Schur^{-1}  & \cL_{\subplus\subplus}^{-1} + \cL_{\subplus\subplus}^{-1} \cLa_{\subplus0} \Schur^{-1} \cLa_{0\subplus} \cL_{\subplus\subplus}^{-1}
\end{pmatrix}.
\]
The invertibility of~$\Schur$ is the crucial element to make the above formal computations rigorous. Two ingredients are used to this end: (i) some coercivity of the symmetric part of the generator on~$\cH_\subplus$, which writes
\[
-\cLs = -\frac12 (\cL+\cL^*) \geq s \Pi_{\subplus} = s(1-\Pi_0).
\]
For Langevin dynamics where $\cLs = \gamma \cL_{\mathrm{FD}}$, this amounts to requiring that~$\kappa(dp)$ satisfies a Poincar\'e inequality; (ii) a property named ``macroscopic coercivity'' in~\cite{DMS15}, namely
\[
\|\cLa_{\subplus 0} \varphi\|_{L^2(\mu)} \geq a \| \Pi_0 \varphi\|_{L^2(\mu)},
\]
and which amounts to the inequality~$\cLa_{\subplus 0}^*\cLa_{\subplus0} \geq a^2 \Pi_0$ in the sense of symmetric operators. For Langevin dynamics, this is equivalent to a Poincar\'e inequality for~$\nu(dq)$, with $a^2 = K_\nu^2/\beta$. In fact, it turns out, for the proof, to be necessary to further decompose~$\cL$ using the projector~$\Pi_{1} = \cLa_{\subplus0} \left(\cLa_{\subplus0}^*\cLa_{\subplus0}\right)^{-1} \cLa_{\subplus0}^*$ onto the range of~$\cLa_{\subplus0} = (1-\Pi_0)\cLa \Pi_0$:
\[
\cL = \begin{pmatrix} 0 & \cLa_{01} & 0 \\
  \cLa_{10} & \cL_{11} & \cL_{12} \\
  0 & \cL_{21} & \cL_{22}
\end{pmatrix},
\qquad
\cLa_{01} = -\cLa_{10}^{*}.
\]
Additional structural assumptions are also needed, namely that there exists an involution~$\cR$ on~$\cH$ such that~$\cR \Pi_{0} = \Pi_{0} \cR = \Pi_{0}$, $\cR \cLs\cR  = \cLs$, and $\cR \cLa \cR = - \cLa$.
Abstract resolvent estimates can then be obtained when the operators $\cLs_{11}$ and $\cL_{21}\cLa_{10}\left(\cLa_{\subplus0}^* \cLa_{\subplus0}\right)^{-1}$ are bounded. The application to Langevin dynamics with a quadratic kinetic energy gives the following result (see Corollary~1 and Proposition~1 in~\cite{BFLS20}). 

\begin{theorem}
  \label{thm:Schur}
  Suppose that~$M = m \mathrm{Id}_\dim$ and that~$V \in C^\infty(\cD)$ satisfies the same assumptions as in Theorem~\ref{th:hypocoercivity}, namely~\eqref{eq:Poincare_inequality} and~\eqref{eq:regularization condition_d}. Then, the operator~$\cL$ is invertible on~$L^2_0(\mu)$ and the following bound holds:
  \begin{equation}
    \label{eq:CL_inv_bound_Schur}
    \left\|\cL^{-1}\right\|_{\cB(L^2_0(\mu))} \leq \frac{2\beta\gamma}{R_\nu} + \frac{8 m}{\gamma} \left( \frac38 + C + \frac{C'}{R_\nu}\right),
  \end{equation}
  where $C$ and $C^{'}$ can be chosen as:
\begin{enumerate}[(i)]
\item If $V$ is convex, then $C=1$ and $C^{'}=0$;
\item If $\nabla_q^2 V \geq -K \mathrm{Id}$ for some $K \geq 0$, then $C=1$ and $C^{'}=K$;
\item In the general case, $C=2$ and $\dps C' = \mathrm{O}(\sqrt{d})$. 
\end{enumerate}
\end{theorem}

The interest of the upper bound~\eqref{eq:CL_inv_bound_Schur} is that it is fully explicit in terms of the parameters of the dynamics (in particular the friction~$\gamma$ and the mass~$m$) and of the dimension~$\dim$ of the system, except for the dependence of the Poincar\'e constant on the dimension and on the potential~$V$. As discussed in~\cite[Section~3.1.4]{BFLS20}, a better scaling $C' = \mathrm{O}(\log d)$ can also be obtained when~$\nu$ satisfies a logarithmic Sobolev inequality and
\[
\forall q \in \R^d, \qquad \left\|\nabla^2V(q)\right\|_{\cB(\ell^2)} \leq c_3\left(1 + |\nabla V(q)|_{\infty}\right).
\]

An interesting observation is that the behavior of the resolvent bounds with respect to~$\gamma$ is consistent with the result of Theorem~\ref{th:hypocoercivity} together with~\eqref{eq:inv_by_int}; it is also sharp. Indeed, for the overdamped limit $\gamma\to +\infty$, we consider the following example:
\[
\cL \Big( p^\top \nabla V + \gamma (V-c_V)\Big) = p^\top M^{-1} \left(\nabla^2 V\right)p - |\nabla V|^2, 
\]
where $c_V$ is a constant chosen such that $p^\top \nabla V + \gamma (V-c_V)$ has a vanishing average with respect to~$\mu$. It is clear that the right hand side is of order~1, while the left-hand side is of order~$\gamma$ when $V$ is not constant. For the limit $\gamma \to 0$, we use the same argument as in~\cite[Proposition~6.3]{HP08}, and consider a function~$\varphi = \phi\circ H$, for which the following equality holds: $\gamma^{-1}\cL \varphi = \cLFD \varphi$.
Here again, the right hand side is of order~1, while the solution to the Poisson equation is of order~$\gamma^{-1}$. These two examples show that there exists $C > 0$ such that
\[
\left\|\cL^{-1}\right\|_{\cB(L^2_0(\mu))} \geq C \max(\gamma,\gamma^{-1}).
\]
In fact, it can be shown that $\cL^{-1}$ is at dominant order equal to $\gamma \cL_{\rm ovd}$ in the overdamped limit $\gamma\to+\infty$ (see~\cite[Theorem~2.5]{LMS16}).

As discussed in~\cite{BFLS20}, the abstract result leading to Theorem~\ref{thm:Schur} for Langevin dynamics can be extended to various other hypocoercive dynamics: Langevin dynamics with non-quadratic kinetic energies~\cite{ST18}, linear Boltzmann/randomized Hybrid Monte Carlo~\cite{BRSS17}, adaptive Langevin dynamics with a variable friction following some Nos\'e--Hoover feedback mechanism~\cite{Herzog18,LSS19}, etc. Some work is however needed to extend the approach to more degenerate dynamics such as generalized Langevin dynamics~\cite{OP11} or chains of oscillators~\cite{Menegaki19}.

\paragraph{Acknowledgements.}
The work of G.S. was funded in part from the European Research Council (ERC) under the European Union's Horizon 2020 research and innovation programme (grant agreement No 810367), and by the Agence Nationale de la Recherche, under grants ANR-19-CE40-0010-01 (QuAMProcs) and ANR-21-CE40-0006 (SINEQ).


\end{document}